\documentclass[12pt]{amsart}
\usepackage{amsmath, amssymb, latexsym, amsthm, amscd, mathrsfs, stmaryrd}
\usepackage{amsfonts, mathrsfs}
\usepackage[title, titletoc, toc]{appendix}
\usepackage[all]{xy}
\usepackage{soul}
\usepackage{tikz}
\usetikzlibrary{arrows}
\usetikzlibrary{calc}
\usepackage{pgfplots}
\pgfplotsset{compat=1.17}
\usepackage{cite}
\usepackage{enumerate}
\usepackage{tikz-cd}
\usepackage{yhmath}

\newtheoremstyle{noparens}%
{}{}%
{\itshape}{}%
{\bfseries}{.}%
{ }%
{\thmname{#1}\thmnumber{ #2}\mdseries\thmnote{ #3}}
\theoremstyle{noparens}
\newtheorem{Def}[subsubsection]{Definition}
\newtheorem{prop}[subsubsection]{Proposition}
\newtheorem{thm}[subsubsection]{Theorem}

\newcommand{\Hom}{\mathrm{Hom}}
\newcommand{\mbf}{\mathbf}

\newcommand{\mbb}{\mathbb}
\newcommand{\mrm}{\mathrm}

\newcommand{\ro}{\mrm{ro}}
\newcommand{\co}{\mrm{co}}

\newcommand{\End}{\mrm{End}}

\usepackage{color}
\newcommand{\nc}{\newcommand}
\nc{\redtext}[1]{\textcolor{red}{#1}}
\nc{\bluetext}[1]{\textcolor{blue}{#1}}
\nc{\greentext}[1]{\textcolor{green}{#1}}
\nc{\yl}[1]{\redtext{From yq: #1}}
\nc{\zb}[1]{\redtext{From zb: #1}}

\title{Mirabolic Howe duality}
\author{Zhaobing Fan}
\address{School of Mathematical Sciences\\ Harbin Engineering University, Harbin, China 150001}
\email{fanzhaobing@hrbeu.edu.cn}
\author{Haitao Ma}
\address{School of Mathematical Sciences\\ Harbin Engineering University, Harbin, China 150001}
\email{hmamath@hrbeu.edu.cn}
\author{Zhicheng Zhang}
\address{School of Mathematical Sciences\\ Harbin Engineering University, Harbin, China 150001}
\email{zhichengzhang@hrbeu.edu.cn}
\thanks{These authors contributed equally to this work.}

\allowdisplaybreaks
\begin{document}
	
	\begin{abstract}
		We establish a duality between a pair of mirabolic quantum groups, i.e., the mirabolic counterpart of quantum Howe duality.
	\end{abstract}
	\maketitle
	
	\section{Introduction}\label{sec1}
	
	Schur-Weyl duality establishes an equivalence between the category of finite-dimensional representations of the symmetric group $S_d$ and that of $GL_n(\mbb{C})$, where $n\geq d$. The dualities for $Sp_{2n}(\mbb{C})$ and $O_n(\mathbb{C})$, were provided in \cite{DDH08}, \cite{D04}, where the corresponding algebra is Brauer algebra. For an affine generalization, see \cite{F21}.
	The duality was also extended to the quantum version by Drinfel'd and Jimbo \cite{D85}, \cite{J86}. Concurrently, the classical work \cite{BLM} presented a geometric construction of quantum groups via the variety of pairs of partial flags, known as the BLM realization. This approach yielded a geometric realization of quantum Schur-Weyl duality for type $A$ \cite{GL}. In \cite{BKLW}, \cite{FL}, the authors generalized BLM realization to $B, C$ and $D$ type partial flag varieties to construct the $\imath$quantum groups and established the Schur-Weyl duality involving $\imath$quantum group. Pouchin demonstrated the Schur-Weyl duality geometrically in a general setting \cite{P}.
	
	The classical Howe duality describes the classical invariant theory with representation-theoretical terms. In the realm of quantum algebra, a parallel development has taken place over the years. Quesen initiated the exploration of quantum Howe duality by demonstrating that the pair  $(U_q(\mathfrak{su}_3), U_q(\mathfrak{u}_2))$ forms a dual pair \cite{Q}. Subsequently, Zhang extended this insight, utilizing the quantum coordinate algebra to establish the quantum Howe duality between $U_q(\mathfrak{gl}_n)$ and $U_q(\mathfrak{gl}_m)$ \cite{Z}. This methodology was further employed to establish the quantum Howe duality between $U_q(\mathfrak{gl}_n)$ and $U_q(\mathfrak{so}_{2n})$, $U_q(\mathfrak{so}_{2n+1})$, $U_q(\mathfrak{sp}_{2n})$, respectively \cite{LZ}. On the other hand, Noumi, Umeda and Wakayama obtained the dual pairs $(U'_q(\mathfrak{so}_{n}), U_q(\mathfrak{sl}_2))$ and $(U'_q(\mathfrak{so}_{n}), U_q(\mathfrak{sp}_2))$, \cite{NUW95}, \cite{NUW96}, where $U'_q(\mathfrak{so}_{n})$ is the $q$-deformation of $U(\mathfrak{so}_n)$ defined by Gavrilik and Klimyk \cite{GK}. On the other hand, Futorny, K\v ri\v zka and Jian Zhang also obtained the dual pairs $(U_q(\mathfrak{sl}_2), U'_{q}(\mathfrak{so}_3))$ and $(U_q(\mathfrak{sl}_2), U_q(\mathfrak{sl}_n))$ by a different approach \cite{FKZ}. Moreover, the quantum skew Howe duality for the $\imath$quantum groups was also established in \cite{ES}, \cite{ST}.
	
	Inspired by BLM realization and its generalizations, Baumann presented a geometric approach of the quantum Howe duality of type $A$ \cite{Ba}(another approach see \cite{W}), Luo and Xu proved the geometric Howe duality of the finite types $A, B, C$, which are coincident with algebraic approach of type $A$ given by \cite{Z} and of type $B/C$ provided by themselves \cite{LX}.
	
	Another direction of generalization of \cite{BLM} is mirabolic setting, which considers the triple of two partial flags and a vector in $V$. The algebra arising from this setting is called mirabolic quantum group.
	In \cite{FZM}, we gave a geometric approach of the duality between the mirabolic quantum group and mirabolic Hecke algebra, called mirabolic Schur-Weyl duality. 
	This work rises a natural question
	\begin{itemize}
		\item Can we achieve the duality between a pair of mirabolic quantum groups, i.e., the mirabolic analogy of Howe duality.
	\end{itemize}
	
	We answer this question in this paper. We prove that the pair of mirabolic quantum groups forms a dual pair based on the geometric approach of mirabolic Schur-Weyl duality. In particular, we provide explicit action formulas of the mirabolic quantum group, which are induced by the action of the mirabolic q-Schur algebras.
	
	This paper is organized as follows. In Section 2, we recall some results in \cite{FZM}. In Section 3, we prove the mirabolic Howe duality.
	
	\section{Quantum Mirabolic Schur-Weyl Duality}
	\subsection{\ }
	
	Let $\mathbb{F}_q$ be a finite field with $q$ elements, $\mathbb{K}$ be any algebraic closed field of characteristic $0$ and $v$ is an indeterminate. Take $V=\mathbb{F}^d_q$ be a vector space of finite dimension $d$ and $G=GL_d(\mathbb{F}_q)$.
	For a fixed positive integer $n$, let $\mathscr{X}_{n,d}$ be the set of all $n$-step partial flags in $V$, i.e., 
	$$
	\mathscr{X}_{n,d}=\{f=(0=V_0\subset V_1\subset \cdots \subset V_{n-1}\subset V_n=V)\}.
	$$
	Similarly, let $\mathscr{Y}$ be the set of all complete flags in $V$, i.e.,
	$$\mathscr{Y}=\{f=(0=V_0\subset V_1\subset \cdots \subset V_{d-1}\subset V_d=V) | \dim{V_{i}/V_{i-1}}=1 \}.$$ 
	
	Let $\mathcal{MS}_{n, d}=\mathbb{K}(v)_{G}(\mathscr{X}_{n,d}\times \mathscr{X}_{n,d}\times V)$ be the space of $G$-invariant $\mathbb{K}(v)$-valued functions on $\mathscr{X}_{n,d}\times \mathscr{X}_{n,d}\times V$. Endow $\mathcal{MS}_{n, d}$ an associative algebra structure by defining convolution product as follows.
	For any $g,h\in \mathcal{MS}_{n, d}$,
	$$
	h\ast g(f,f',v)=\sum_{(f'',v)\in \mathscr{X}_{n,d}\times V}h(f,f'',w)g(f'',f',v-w)
	.$$
	Similarly, we obtain a convolution algebra $\mathcal{MH}=\mathbb{K}(v){G}(\mathscr{Y}\times \mathscr{Y}\times V)$, which is isomorphic to the mirabolic Hecke algebra (see \cite{S} and \cite{R14}), and a space $\mathcal{MV}_{n|d}=\mathbb{K}(v)_{G}(\mathscr{X}_{n,d}\times \mathscr{Y}\times V)$. Then the space $\mathcal{MV}_{n, d}$ has a natural left $\mathcal{MS}_{n,d}$-action and right $\mathcal{MH}$-action by convolution product.
	
	In \cite{FZM}, we established the duality between $\mathcal{MS}_{n,d}$ and $\mathcal{MH}$ as follows.
	\begin{thm}\label{msw-s}
		For $n\geq d$, we have the following double centralizer property,
		\begin{align*}
			\mathcal{MS}_{n,d}\cong \End_{\mathcal{MH}}{\mathcal{MV}_{n,d}}, \mathcal{MH}\cong \End_{\mathcal{MS}_{n,d}}{\mathcal{MV}_{n,d}}.
		\end{align*}
	\end{thm}
	
	Recall the mirabolic quantum group $\mathbf{MU}_n$ introduced in \cite{FZM}.
	\begin{Def}
		Fix a positive integer $n$. The algebra $\mathbf{MU}_n$ is a $\mathbb{K}(v)$-algebra generated by $\mathbf{E}_i, \mathbf{F}_i, \mathbf{H}^{\pm}_a, \mathbf{L}$ for $i\in [1, n-1]$ and $a\in [1, n]$, with the following relations.
		\begin{align*}
			& \mathbf{H}_{a}\mathbf{H}^{-1}_{a}=1, \mathbf{H}_{a}\mathbf{H}_{b}=\mathbf{H}_{b} \mathbf{H}_{a};\\
			& \mathbf{E}_{i}^2\mathbf{E}_{j}+\mathbf{E}_{j}\mathbf{E}_{i}^2=(v+v^{-1})\mathbf{E}_{i}\mathbf{E}_{j}\mathbf{E}_{i} \quad |i-j|=1;\\
			& \mathbf{F}_{i}^{2}\mathbf{F}_{j}+\mathbf{F}_{j}\mathbf{F}_{i}^2=(v+v^{-1})\mathbf{F}_{i}\mathbf{F}_{j}\mathbf{F}_{i} \quad |i-j|=1;\\
			& \mathbf{E}_{i}\mathbf{E}_{j}=\mathbf{E}_{j}\mathbf{E}_{i},  \mathbf{F}_{i}\mathbf{F}_{j}=\mathbf{F}_{j}\mathbf{F}_{i} \quad |i-j|>1;\\
			& \mathbf{H}_{a}\mathbf{E}_{i}=v^{\delta_{a i}-\delta_{a, i+1}}\mathbf{E}_{i}\mathbf{H}_{a};\\
			& \mathbf{H}_{a}\mathbf{F}_{i}=v^{-\delta_{a i}+\delta_{a, i+1}}\mathbf{F}_{i}\mathbf{H}_{a};\\
			& \mathbf{E}_{i}\mathbf{F}_{j}-\mathbf{F}_{j}\mathbf{E}_{i}=\delta_{i, j}\frac{\mathbf{H}_{i}\mathbf{H}^{-1}_{i+1}-\mathbf{H}^{-1}_{i}\mathbf{H}_{i+1}}{(v-v^{-1})};\\
			& \mathbf{H}_{a}\mathbf{L}=\mathbf{L}\mathbf{H}_{a}, \mathbf{L}^2=\mathbf{L};\\
			& \mathbf{L}\mathbf{E}_{i}=\mathbf{L}\mathbf{E}_{i}\mathbf{L}, \mathbf{L}\mathbf{F}_{i}=\mathbf{L}\mathbf{F}_{i}\mathbf{L};\\
			& \frac{v^2-v^{-2}}{v-v^{-1}} \mathbf{E}_{i}\mathbf{L}\mathbf{E}_{i}=v^{-1}\mathbf{E}_{i}^{2}\mathbf{L}+v\mathbf{L}\mathbf{E}_{i}^{2};\\
			& \frac{v^2-v^{-2}}{v-v^{-1}}\mathbf{F}_{i}\mathbf{L}\mathbf{F}_{i}=v\mathbf{F}_{i}^{2}\mathbf{L}+v^{-1}\mathbf{L}\mathbf{F}_{i}^{2}.
		\end{align*}
	\end{Def}
	
	There is a surjective algebra homomorphism $\kappa_n: \mathbf{MU}_n \to \mathcal{MS}_{n,d}$, and we prove the following result \cite{FZM}.
	\begin{thm}\label{msw-q}
		For $n\geq d$, we have the following double centralizer property,
		\begin{align*}
			\mathbf{MU}_{n}\to \End_{\mathcal{MH}}{\mathcal{MV}_{n,d}}\text{ is a surjective, }\mathcal{MH}\cong \End_{\mathbf{MU}_{n}}{\mathcal{MV}_{n,d}}.
		\end{align*}
	\end{thm}
	
	\section{Mirabolic Howe duality}
	\subsection{Space $\mathcal{MV}_{n|m}$}
	
	Fix two positive integers $n\geq m$. Consider the space of $G$-invariant $\mathbb{K}(v)$-valued functions on $\mathscr{X}_{n,d}\times \mathscr{X}_{m,d}\times V$, which is denoted by $\mathcal{MV}_{n|m}$.
	
	\begin{Def}\label{def 3.1.1}
		A decorated matrix is a pair $(A, \Delta)$, where $A=(a_{i j})\in Mat_{n\times m}(\mathbb{N})$ 
		and $\Delta =\{(i_1, j_1), \cdots, (i_k, j_k)\}$ is a set that satisfies
		$$
		1\leq  i_1< \cdots < i_k \leq  n , \ 1\leq j_k< \cdots< j_1 \leq m, 
		$$
		with the additional condition that for all $(i, j)\in \Delta$, the entry $a_{i j}>0$. We denote by $\Xi_{n|m}$ the set of decorated matrices and $\Xi_{n|m,d}=\{(A, \Delta)\in \Xi_{n|m} | \sum_{i j}a_{i, j}=d \}$.
	\end{Def}
	
	There is a bijection between $G$-orbits on $\mathscr{X}_{n,d}\times \mathscr{X}_{m,d}\times V$ and $\Xi_{n|m,d}$ \cite{MWZ}.
	The orbit that corresponding to $(A,\Delta)\in \Xi_{n|m,d}$ is denoted by $\mathcal{O}_{A,\Delta}$ and $e_{A,\Delta}$ is its characteristic function. 
	
	For any $(A,\Delta)\in \Xi_{n|m, d}$, define $[A]_{\Delta}=v^{-d(A,\Delta)+r(A,\Delta)}e_{A,\Delta}$, where 
	$$-d(A,\Delta)+r(A,\Delta)=-\sum\limits_{i<k\, or\,j< l}a_{i j}a_{kl}-\sum\limits_{\{(i, j)\}\leq \Delta} a_{i j}$$ 
	and $\{(i, j)\}\leq \Delta$ means there is $(k,l)\in \Delta$ such that $i\leq k$ and $j\leq l$.
	Then the set $\{[A]_\Delta\ |\ (A,\Delta)\in \Xi_{n|m,d}\}$ form a basis of $\mathcal{MV}_{n|m}$.
	
	\begin{prop}
		The dimension of $\mathcal{MV}_{n|m}$ is
		$$
		\sum^{\min\{m,d\}}\limits_{l=0}\binom{m}{l}\binom{n}{l}\binom{nm+d-1-l}{d-l}.
		$$
	\end{prop}
	\begin{proof}
		For any fixed integer $\min\{m,d\}\geq l\geq 0$, we define subset $X_l=\{[A]_{\Delta}\ |\ \sharp \Delta=l\}\subset \mathcal{MV}_{n|m}$, then we have 
		$$\dim \mathcal{MV}_{n|m}=\sum^{m}\limits_{l=0}\sharp X_l.$$
		
		For a fixed $l$, the $\sharp X_l$ is actually to count the number of nm-step partitions of $d$ such that there are at least $l$-th entries are non-zero, and they lie in different columns and rows. Therefore, we have 
		$$
		\sharp X_l=\binom{m}{l}\binom{n}{l}\binom{nm+d-l-1}{d-l},
		$$
		and the dimension of $\mathcal{MV}_{n|m}$ is as desire.
	\end{proof}
	\subsection{Explicit action}
	
	Similar with $\mathcal{MV}_{n|m}$, the set $\{[A]_\Delta\ |\ (A,\Delta)\in \Xi_{n|n,d}\}$ form a basis of $\mathcal{MS}_{n,d}$.
	For $i\in [1, n-1]$ and $a\in [1,n]$, we define
	$$E_i=\sum[B]_{\emptyset},\ F_i=\sum[C]_{\emptyset},\ H_a^{\pm}=\sum v^{\mp d_{a, a}}[D]_{\emptyset},\  L=\sum v^{-2d_{1,1}}[D]_{\emptyset}+\sum v^{-d'_{1,1}}[D']_{\{(1,1)\}},$$
	where $(B,\emptyset)$, $(C,\emptyset)$ $(D,\emptyset)$ and $(D',\{(1,1)\})$ run over all matrices in $\Xi_{n|n,d}$ such that $D$, $D'$, $B-E_{i,i+1}$, and $C-E_{i+1,i}$ are diagonal matrices, respectively. Then we have the following proposition in \cite{FZM}.
	
	\begin{prop}\label{prop gnrt}
		The elements $E_i, F_i, H^{\pm}_{a}$ and $L$ are the generators of $\mathcal{MS}_{n,d}$, where $i\in [1, n-1]$ and $a\in [1, n]$.
	\end{prop}
	
	For any two positive integers $N, t$, let $$\overline{[N, t]}_v=\prod_{i\in [1, t]}\frac{v^{-2(N-i+1)}-1}{v^{-2i}-1}.$$
	It is clear that $\mathcal{MV}_{n|m}$ has a left $\mathcal{MS}_{n,d}$-action $\Phi$ and a right $\mathcal{MS}_{m,d}$-action $\Psi$ by convolution product. Then we have the following propositions.
	
	\begin{prop}\label{prop 2.2.1}
		Assume $h\in [1, n-1]$ and $r\in [1,n]$. For $(A,\Delta)\in \Xi_{n|m,d}$ where $\Delta = \{(i_1, j_1), \cdots, (i_k, j_k)\}$, let $\beta_A(p)=\sum_{j\geq p}a_{h, j}-\sum_{j>p}a_{h+1, j}$ and $\beta_A'(p)=\sum_{j \leq p}a_{h+1, j}-\sum_{j<p}a_{h, j}$ for $p\in [1,m]$, then the left $\mathcal{MS}_{n,d}$-action is given explicitly as follows.
		
		\noindent $(a)$ For any $r\in [1,n]$, then 
		$$
		H^{\pm}_r\ast [A]_{\Delta}=v^{\mp \sum_{j\in [1,m]}a_{r,j}}[A]_{\Delta}.
		$$
		
		\noindent$(b)$ If $i_1>1$, then
		\begin{align*}
			\relax L\ast [A]_{\Delta}=v^{-2\sum_{j>j_1}a_{1,j}}([A]_{\Delta}+\sum_{t>j_1, a_{1, t}>0}v^{\sum_{j_1<j\leq t}a_{1,j}}[A]_{\Delta_t})
		\end{align*}
		where 
		\begin{align*}
			\Delta_t=\{(1, t), (i_1,j_1), \cdots, (i_k, j_k)\}.
		\end{align*}
		
		\noindent$(c)$ If $i_1=1$, then 
		\begin{align*}
			\relax L\ast [A]_{\Delta}=&v^{-2\sum_{j>j_1}a_{1,j}}(\sum_{t>j_2, a_{1, t}>0}v^{-\sum_{j\leq j_1}a_{1,j}+\sum_{j\leq t}a_{1,j}}(1-v^{-2a_{1,j_1}})[A]_{\Delta_t})\\
			&+v^{-2\sum_{j>j_1}a_{1,j}}(v^{-\sum_{j_2<j\leq j_1}a_{1,j}}(1-v^{-2a_{1, j_1}})[A]_{\Delta\setminus \{(i_1, j_1)\}})
		\end{align*}
		where \begin{align*}
			\Delta_t=\{(1, t), (i_2,j_2), \cdots, (i_{k}, j_{k})\}.
		\end{align*}
		
		\noindent$(d)$ If $\Delta=\emptyset$, then
		\begin{align*}
			\relax L\ast [A]_{\Delta}=\sum\limits_{t\in [1, m], a_{1, t}>0} v^{-\sum_{j\leq m}a_{1,j}-\sum_{j>t}a_{1,j}} [A]_{\{(1, t)\}}+v^{-2\sum_{j\leq m}a_{1, j}}[A]_{\emptyset}.
		\end{align*}
		
		\noindent$(e)$ If for any $t$, $i_t \neq h, h+1$, then
		$$E_h\ast [A]_\Delta = \sum\limits_{p\in[1, m], a_{h+1, p} \geq 1} v^{\beta_A(p)}\overline{[a_{h, p}+1, 1]}_v[A+\mathbf{E}_{h, p}-E_{h+1, p}]_{\Delta};
		$$
		$$F_h\ast [A]_\Delta = \sum\limits_{p\in[1, m], a_{h, p} \geq 1} v^{\beta_A'(p)}\overline{[a_{h+1, p}+1, 1]}_v[A+E_{h+1, p}-E_{h, p}]_{\Delta}.$$
		\noindent$(f)$ If there exists $l$ such that $i_l = h, i_{l+1} \neq h+1$, then
		\begin{align*}
			\relax E_h\ast [A]_{\Delta} = &
			\sum\limits_{p \in [j_{l+1}+1, j_l-1], a_{h+1, p} \geq 1}v^{\beta_A(p)-1}\overline{[a_{h, p}+1, 1]}_v[A+E_{h, p}-E_{h+1, p}]_{\Delta} \\
			& + \sum\limits_{p = j_l, a_{h+1, p} \geq 1} v^{\beta_A(p)-1}\overline{[a_{h, p}, 1]}_v[A+E_{h, p}-E_{h+1, p}]_{\Delta}\\
			&+ \sum\limits_{p \notin [j_{l+1}+1, j_l], a_{h+1, p} \geq 1}v^{\beta_A(p)}\overline{[a_{h, p}+1, 1]}_v[A+E_{h, p}-E_{h+1, p}]_{\Delta};
		\end{align*}
		\begin{align*}
			\relax F_h\ast [A]_{\Delta}= & \sum\limits_{i_{l+1}<p\leq i_l, a_{h, p} \geq 1} v^{\beta_A'(p)-1}\overline{[a_{h+1, p}+1, 1]}_v[A+E_{h+1, p}-E_{h, p}]_{\Delta} \\  
			& \sum\limits_{p\notin [i_{l+1}+1,i_l], a_{h, p} \geq 1} v^{\beta_A'(p)}\overline{[a_{h+1, p}+1, 1]}_v[A+E_{h+1, p}-E_{h, p}]_{\Delta} \\    
			& + \sum\limits_{p = j_l, a_{h, p} \geq 1}v^{\sum_{j \leq j_{l+1}}a_{h+1, j}-\sum_{j< j_{l}}a_{h, j}}[A+E_{h+1, p}-E_{h, p}]_{\Delta_{j_l}} \\
			& + \sum\limits_{j_{l+1}<p = t<j_{l}, a_{h, p} \geq 1}v^{\sum_{j \leq j_{l+1}}a_{h+1, j}-\sum_{j< t}a_{h, j}}[A+E_{h+1, p}-E_{h, p}]_{\Delta_t},
		\end{align*}
		where
		\begin{align*}
			\Delta_{j_l} &= \{(i_1, j_1), \cdots, (i_{l-1}, j_{l-1}), (h+1, j_l), (i_{l+1}, j_{l+1}), \cdots, (i_k, j_k)\}, \\
			\Delta_t &= \{(i_1, j_1), \cdots, (i_{l-1}, j_{l-1}), (h, j_l), (h+1, t), (i_{l+1}, j_{l+1}), \cdots, (i_k, j_k)\}.
		\end{align*}
		\noindent $(g)$ If there exists $l$ such that $i_{l-1} \neq h, i_{l} = h+1$, then
		\begin{align*}
			\relax E_h\ast [A]_{\Delta} = &\sum\limits_{p \in [1, m], a_{h+1, p} \geq 1} v^{\beta_A(p)}\overline{[a_{h, p}+1, 1]}_v[A+E_{h, p}-E_{h+1, p}]_{\Delta} \\
			& + \sum\limits_{p=j_l, a_{h+1, p} \geq 1} v^{\beta_A(p)-\sum_{j_{l+1}< j \leq p}a_{h+1, j}+1}[A+E_{h, p}-E_{h+1, p}]_{\Delta_{j_l}}\\
			& +\sum\limits_{\substack{p= j_l, a_{h+1, p}\geq 1,\\ j_{l+1}<t< j_l}}v^{\beta_A(p)-\sum_{t<j \leq p}a_{h+1, j}+1}[A+E_{h, p}-E_{h+1, p}]_{\Delta_t},
		\end{align*}
		where
		\begin{align*}
			\Delta_{j_l}&= \{(i_1, j_1), \cdots, (i_{l-1}, j_{l-1}), (h, j_l), (i_{l+1}, j_{l+1}), \cdots, (i_k, j_k)\},\\
			\Delta_t& = \{(i_1, j_1), \cdots, (i_{l-1}, j_{l-1}), (h, j_l), (h+1, t), (i_{l+1}, j_{l+1}), \cdots, (i_k, j_k)\};
		\end{align*}
		\begin{align*}
			\relax F_h\ast [A]_{\Delta} = & \sum\limits_{p \neq j_l, a_{h, p} \geq 1} v^{\beta_A'(p)}\overline{[a_{h+1, p}+1, 1]}_v[A+E_{h+1, p}-E_{h, p}]_{\Delta} \\
			& + \sum\limits_{p=j_l, a_{h, p} \geq 1} v^{\beta_A'(p)}\overline{[a_{h+1, p}, 1]}_v[A+E_{h+1, p}-E_{h, p}]_{\Delta}.
		\end{align*}
		\noindent$(h)$ If there exists $l$ such that $i_l=  h, i_{l+1} = h+1$, then
		\begin{align*}
			& \relax E_h\ast [A]_{\Delta}\\
			= & \sum\limits_{p \notin[j_{l+1}+1, j_l], a_{h+1, p} \geq 1}v^{\beta_A(p)}\overline{[a_{h, p}+1, 1]}_v[A+E_{h, p}-E_{h+1, p}]_{\Delta} \\
			& +  \sum\limits_{j_{l+1}<p<j_{l}, a_{h+1, p} \geq 1}v^{\beta_A(p)-1}\overline{[a_{h, p}+1, 1]}_v[A+E_{h, p}-E_{h+1, p}]_{\Delta} \\
			& +  \sum\limits_{p=j_l, a_{h+1, p} \geq 1} v^{\beta_A(p)-1}\overline{[a_{h, p}, 1]}_v[A+E_{h, p}-E_{h+1, p}]_{\Delta}\\
			& + \sum\limits_{p=j_{l+1}, a_{h+1, p} \geq 1}v^{\beta_A(p)-\sum_{j_{l+2}<j \leq p}a_{h+1, j}+1}(1-v^{-2})\overline{[a_{h, p}+1, 1]}_v[A+E_{h, p}-E_{h+1, p}]_{\Delta_{j_{l+1}}}\\
			& + \sum\limits_{\substack{p=j_{l+1}, a_{h+1, p} \geq 1,\\ j_{l+2}<t<j_{l+1}}} v^{\beta_A(p)-\sum_{t<j \leq p}a_{h+1, j}+1}(1-v^{-2})\overline{[a_{h, p}+1, 1]}_v[A+E_{h, p}-E_{h+1, p}]_{\Delta_t},
		\end{align*}
		where
		\begin{align*}
			\Delta_{j_{l+1}}&= \{(i_1, j_1), \cdots, (i_{l-1}, j_{l-1}), (h, j_l), (i_{l+2}, j_{l+2}), \cdots, (i_k, j_k)\}, \\
			\Delta_t &= \{(i_1, j_1), \cdots, (i_{l-1}, j_{l-1}), (h, j_l), (h+1, t), (i_{l+2}, j_{l+2}), \cdots, (i_k, j_k)\};
		\end{align*}
		\begin{align*}
			& \relax F_h\ast [A]_{\Delta} \\
			=  & \sum\limits_{p \neq j_{l+1}, a_{h, p} \geq 1} v^{\beta_A'(p)}\overline{[a_{h+1, p}+1, 1]}_v[A+E_{h+1, p}-E_{h, p}]_{\Delta} \\
			& +\sum\limits_{p=j_{l+1}, a_{h, p} \geq 1}v^{\beta_A'(p)}\overline{[a_{h+1, p}, 1]}_v[A+E_{h+1, p}-E_{h, p}]_{\Delta} \\
			& + \sum\limits_{p=j_{l}, a_{h, p} \geq 1} v^{\beta_A'(j_{l+1})-\sum_{j_{l+1}\leq j< p}a_{h, j}}(1-v^{-2})\overline{[a_{h+1, j_{l+1}}, 1]}_v[A+E_{h+1, p}-E_{h, p}]_{\Delta_{j_l}} \\ 
			& + \sum\limits_{j_{l+1}<t<j_l, a_{h, t} \geq 1} v^{\beta_A'(j_{l+1})-\sum_{j_{l+1}\leq j < t}a_{h, j}}(1-v^{-2})\overline{[a_{h+1, j_{l+1}}, 1]}_v[A+E_{h+1, t}-E_{h, t}]_{\Delta_t},
		\end{align*}
		where
		\begin{align*}
			\Delta_{j_l} &= \{(i_1, j_1), \cdots, (i_{l-1}, j_{l-1}), (h+1, j_l), (i_{l+2}, j_{l+2}), \cdots, (i_k, j_k)\},\\
			\Delta_t &=\{(i_1, j_1), \cdots, (i_{l-1}, j_{l-1}), (h, j_l), (h+1, t), (i_{l+2}, j_{l+2}), \cdots, (i_k, j_k)\}.
		\end{align*}
	\end{prop}
	\begin{proof}
		These identities follow from Proposition 4.1.5, Proposition 4.1.6 and Proposition 4.1.7 in \cite{FZM}. 
	\end{proof}
	
	Next we show the right action of $\mathcal{MS}_{m,d}$ on $\mathcal{MV}_{n|m}$.
	\begin{prop}\label{prop 2.2.4}
		Assume $h\in [1, m-1]$ and $r\in [1,m]$. For $(A,\Delta)\in \Xi_{n|m,d}$, where $\Delta = \{(i_1, j_1), \cdots, (i_k, j_k)\}$, let $\xi_A(p)=\sum_{j\leq p}a_{j, h+1}-\sum_{j<p}a_{j, h}$ and $\xi_A'(p)=\sum_{j\geq p}a_{j, h}-\sum_{j>p}a_{j, h+1}$ for $p\in [1,n]$. Then the following identities give the explicit right $\mathcal{MS}_{m,d}$-action. 
		
		\noindent$(a)$ For any $r\in [1,n]$,
		\begin{align*}
			\relax [A]_{\Delta}\ast H^{\pm}_r=v^{\mp\sum_{i\in [1,n]}a_{i,r}}[A]_{\Delta}.
		\end{align*}
		\noindent$(b)$ If $1<j_k$, then 
		\begin{align*}
			\relax [A]_{\Delta}\ast L=v^{-2\sum_{i>i_k}a_{i,1}}([A]_{\Delta}+\sum_{t>i_k, a_{t,1}>0}v^{\sum_{i_k<t\leq t}a_{i,1}}[A]_{\Delta_t})
		\end{align*}
		where 
		\begin{align*}
			\Delta_t=\{(i_1,j_1), \cdots, (i_k, j_k), (t,1)\}.
		\end{align*}
		
		\noindent$(c)$ If $j_k=1$, then
		\begin{align*}
			\relax [A]_{\Delta}\ast L=&v^{-2\sum_{i>i_{k}}a_{i,1}}(\sum_{t>i_{k-1}, a_{t,1}>0}v^{-\sum_{i\leq i_k}a_{i,1}+\sum_{i\leq t}a_{i,1}}(1-v^{-2a_{i_k,1}})[A]_{\Delta_t})\\
			&+v^{-2\sum_{i> i_k}a_{i,1}}(v^{-\sum_{i_{k-1}<\leq i_k}a_{i,1}}(1-v^{-2a_{i_k,1}})[A]_{\Delta\setminus \{(i_k, j_k)\}})
		\end{align*}
		where \begin{align*}
			\Delta_t=\{(i_1,j_1), \cdots, (i_{k-1}, j_{k-1}), (t, j_{k})\}.
		\end{align*}
		
		\noindent$(d)$ If $\Delta=\emptyset$, we have 
		\begin{align*}
			\relax [A]_{\emptyset}\ast L=\sum\limits_{t\in [1,n], a_{t,1}>0} v^{-\sum_{i\leq n}a_{i,1}-\sum_{i>t}a_{i,1}} [A]_{\{(t, 1)\}}+v^{-2\sum_{i\leq n}a_{i, 1}}[A]_{\emptyset}.
		\end{align*}
		
		\noindent $(e)$ For any $t$, $j_t \neq h, h+1$, then
		$$[A]_\Delta \ast E_h= \sum\limits_{p\in[1, n], a_{p, h} \geq 1} v^{\xi_A(p)}\overline{[a_{p, h+1}+1, 1]}_v[A+E_{p, h+1}-E_{p, h}]_{\Delta};$$
		$$[A]_\Delta \ast F_h = \sum\limits_{p\in[1, n], a_{p, h+1} \geq 1}v^{\xi_A'(p)}\overline{[a_{p, h}+1, 1]}_v[A+E_{p, h}-E_{p, h+1}]_{\Delta}.$$
		
		\noindent$(f)$ There exists $l$ such that $j_l = h, j_{l-1} \neq h+1$, then
		\begin{align*}
			\relax [A]_{\Delta}\ast E_h = & \sum\limits_{p \in[i_{l-1}+1, i_l], a_{p, h} \geq 1} v^{\xi_A(p)-1}\overline{[a_{p, h+1}+1, 1]}_v[A+E_{p, h+1}-E_{p, h}]_{\Delta} \\
			& +\sum\limits_{p \notin [i_{l-1}+1, i_l], a_{p, h} \geq 1} v^{\xi_A(p)}\overline{[a_{p, h+1}+1, 1]}_v[A+E_{p, h+1}-E_{p, h}]_{\Delta} \\  
			& + \sum\limits_{p = i_l, a_{p, h} \geq 1}v^{\sum_{i\leq  i_{l-1}}a_{i, h+1}-\sum_{i< i_{l}}a_{i, h}}[A+E_{p, h+1}-E_{p, h}]_{\Delta_{i_l}} \\
			& + \sum\limits_{i_{l-1}<p =s<i_{l}, a_{p, h} \geq 1}v^{\sum_{i \leq i_{l-1}}a_{i, h+1}-\sum_{i< s}a_{i, h}}[A+E_{p, h+1}-E_{p, h}]_{\Delta_s}, 
		\end{align*}
		where
		\begin{align*}
			\Delta_{i_l} &= \{(i_1, j_1), \cdots, (i_{l-1}, j_{l-1}), (i_l, h+1), (i_{l+1}, j_{l+1}), \cdots, (i_k, j_k)\}, \\
			\Delta_s &= \{(i_1, j_1), \cdots, (i_{l-1}, j_{l-1}), (s, h+1), (i_{l}, h), (i_{l+1}, j_{l+1}), \cdots, (i_k, j_k)\};
		\end{align*}
		\begin{align*}
			\relax [A]_{\Delta}\ast F_h = &
			\sum\limits_{p \in [i_{l-1}+1, i_l-1], a_{p, h+1} \geq 1}v^{\xi_A'(p)-1}\overline{[a_{p, h}+1, 1]}_v[A+E_{p, h}-E_{p, h+1}]_{\Delta} \\
			& + \sum\limits_{p = i_l, a_{p, h+1} \geq 1} v^{\xi_A'(p)-1}\overline{[a_{p, h}, 1]}_v[A+E_{p, h}-E_{p, h+1}]_{\Delta}\\
			&+ \sum\limits_{p \notin [i_{l-1}+1, i_l], a_{p, h+1} \geq 1}v^{\xi_A'(p)}\overline{[a_{p, h}+1, 1]}_v[A+E_{p, h,}-E_{p, h+1}]_{\Delta}.
		\end{align*}
		
		\noindent $(g)$ There exists $l$ such that $j_{l+1} \neq h, j_{l} = h+1$, then
		\begin{align*}
			\relax [A]_{\Delta}\ast E_h  = & \sum\limits_{p \neq i_l, a_{p, h} \geq 1} v^{\xi_A(p)}\overline{[a_{p, h+1}+1, 1]}_v[A+E_{p, h+1}-E_{p, h}]_{\Delta} \\
			& + \sum\limits_{p=i_l, a_{p, h} \geq 1} v^{\xi_A(p)}\overline{[a_{p, h+1}, 1]}_v[A+E_{p, h+1}-E_{p, h}]_{\Delta};
		\end{align*}
		\begin{align*}
			\relax [A]_{\Delta}\ast F_h = &\sum\limits_{p \in [1, n], a_{p, h+1} \geq 1} v^{\xi_A'(p)}\overline{[a_{p, h}+1, 1]}_v[A+E_{p, h}-E_{p, h+1}]_{\Delta} \\
			& + \sum\limits_{p=i_l, a_{p, h+1} \geq 1} v^{\xi_A(p)-\sum_{i_{l-1}< i \leq p}a_{i, h+1}+1}[A+E_{p, h}-E_{p, h+1}]_{\Delta_{i_l}}\\
			& +\sum\limits_{p= i_l, a_{p, h+1}\geq 1, i_{l-1}<s< i_l}v^{\xi_A'(p)-\sum_{s<i \leq p}a_{i, h+1}+1}[A+E_{p, h}-E_{p, h+1}]_{\Delta_s},
		\end{align*}
		where
		\begin{align*}
			\Delta_{i_l}&= \{(i_1, j_1), \cdots, (i_{l-1}, j_{l-1}), (i_l, h), (i_{l+1}, j_{l+1}), \cdots, (i_k, j_k)\},\\
			\Delta_s& = \{(i_1, j_1), \cdots, (i_{l-1}, j_{l-1}), (s, h+1), (i_l, h), (i_{l+1}, j_{l+1}), \cdots, (i_k, j_k)\}.
		\end{align*}
		\noindent $(h)$ There exists $l$ such that $j_l= h, j_{l-1} = h+1$, then
		\begin{align*}
			& \relax [A]_{\Delta}\ast E_h\\ 
			= & \sum\limits_{p\neq i_{l-1}, a_{p, h} \geq 1}v^{\xi_A(p)}\overline{[a_{p, h+1}+1, 1]}_v[A+E_{p, h+1}-E_{p, h}]_{\Delta} \\
			& +  \sum\limits_{p=i_{l-1}, a_{p, h} \geq 1}v^{\xi_A(p)}\overline{[a_{p, h+1}, 1]}_v[A+E_{p, h+1}-E_{p, h}]_{\Delta} \\
			& + \sum\limits_{p=i_{l}, a_{p, h} \geq 1} v^{\xi_A(i_{l-1})-\sum_{i_{l-1}\leq j< p}a_{j, h}}(1-v^{-2})\overline{[a_{i_{l-1}, h+1}, 1]}_v[A+E_{p, h+1}-E_{p, h}]_{\Delta_{i_l}} \\ 
			& + \sum\limits_{i_{l-1}<s<i_l, a_{s, h} \geq 1} v^{\xi_A(j_{l-1})-\sum_{i_{l-1}\leq j< s}a_{j, h}}(1-v^{-2})\overline{[a_{i_{l-1}, h+1}, 1]}_v[A+E_{s, h+1}-E_{s, h}]_{\Delta_s},
		\end{align*}
		where
		\begin{align*}
			\Delta_{i_l} &= \{(i_1, j_1), \cdots, (i_{l-2}, j_{l-2}), (i_l, h+1), (i_{l+1}, j_{l+1}), \cdots, (i_k, j_k)\},\\
			\Delta_s &=\{(i_1, j_1), \cdots, (i_{l-2}, j_{l-2}), (s, h+1), (i_l, h), (i_{l+1}, j_{l+1}), \cdots, (i_k, j_k)\};
		\end{align*}
		\begin{align*}
			& \relax [A]_{\Delta}\ast F_h \\
			= & \sum\limits_{p \notin[i_{l-1}+1, i_l], a_{p, h+1} \geq 1}v^{\xi_A'(p)}\overline{[a_{p, h}+1, 1]}_v[A+E_{p, h}-E_{p, h+1}]_{\Delta} \\
			& +  \sum\limits_{i_{l-1}<p<i_{l}, a_{p, h+1} \geq 1}v^{\xi_A'(p)-1}\overline{[a_{p, h}+1, 1]}_v[A+E_{p, h}-E_{p, h+1}]_{\Delta} \\
			& +  \sum\limits_{p=i_l, a_{p, h+1} \geq 1} v^{\xi_A(p)-1}\overline{[a_{p, h}, 1]}_v[A+E_{p, h}-E_{p, h+1}]_{\Delta}\\
			& + \sum\limits_{p=i_{l-1}, a_{p, h+1} \geq 1}v^{\xi_A'(p)-\sum_{i_{l-2}<i \leq p}a_{i, h+1}+1}(1-v^{-2})\overline{[a_{p, h}+1, 1]}_v[A+E_{p, h}-E_{p, h+1}]_{\Delta_{i_{l-1}}}\\
			& + \sum\limits_{\substack{p=i_{l-1}, a_{p, h+1} \geq 1,\\ i_{l-2}<s<i_{l-1}}} v^{\xi_A(p)-\sum_{s<i \leq p}a_{i, h+1}-1}(1-v^{-2})\overline{[a_{p, h}+1, 1]}_v[A+E_{p, h}-E_{p, h+1}]_{\Delta_s},
		\end{align*}
		where
		\begin{align*}
			\Delta_{i_{l-1}} &= \{(i_1, j_1), \cdots, (i_{l-2}, j_{l-2}), (i_l, h), (i_{l+1}, j_{l+1}), \cdots, (i_k, j_k)\}, \\
			\Delta_s &= \{(i_1, j_1), \cdots, (i_{l-2}, j_{l-2}), (s, h+1), (i_l, h), \cdots, (i_k, j_k)\}.\\
		\end{align*}
	\end{prop}
	
	
	\begin{proof}
		It is sufficient to prove the formulas at the specialization of $v=\sqrt{q}$.
		There is an anti-isomorphism $^t:\mathcal{MS}_{m,d}\to \mathcal{MS}_{m,d}$, sending $e_{A,\Delta}$ to $e_{A^t, \Delta^t}$,  where $A^t$ denotes the transpose matrix and $\Delta^t=\{(j_k, i_k), \cdots (j_1, i_1)\}$, then $(e_{A,\Delta}\ast e_{A^{'}, \Delta^{'}})^t=e_{A^{'t}, \Delta^{'t}} \ast e_{A^t, \Delta^t}$. 
		We prove $(h)$ for the right $ \mathcal{MS}_{m,d}$-action $\Psi$, and other identities are similar.
		\begin{align*}
			[A]_{\Delta}\ast E_h=&(F_h\ast [A^{t}]_{\Delta^t})^t
			= ([C]_{\emptyset}\ast [A^{t}])^{t}\\
			=& (v^{-\sum_{j\leq n}a_{j,h+1}}e_{C,\emptyset}\ast v^{-\sum_{i<k,j\geq l}a_{i,j}a_{k,l}+\sum_{(i,j)\leq \Delta^t}a_{i,j}}e_{A^t,\Delta^t})^{t},
		\end{align*}
		where $(C,\emptyset)\in \Xi_{m|m,d}$ such that $\co(C)=\ro(A^t)$ and $C-E_{h+1,h}$ is a diagonal matrix.
		
		By definition, we have the following identity.
		\begin{align*}
			e_{C, \emptyset}\ast e_{A^t, \Delta^t}=\sum\limits_{p\in [1,n], a_{h,p}>0}\sharp W_p e_{A^t-E_{i,p}+E_{i+1,p},\Delta'},
		\end{align*}
		where $W_p$ is the set of all subspaces $T\subset V$ satisfying the following conditions.
		\begin{enumerate}
			\item $(f,f',\omega)$ is a fixed triple in $\mathscr{X}\times \mathscr{Y}\times V$ lying in the orbit $\mathcal{O}_{A^t-E_{h,p}+E_{h+1,p},\Delta'}$,
			\item $V_{h}\subset T\subset V_{h+1}$ and $\dim{T}=\dim{V_h}+1$,
			\item $ V_h\cap V'_j=T\cap V'_j (j<p), \ V_h\cap V'_{j}\neq T\cap V'_j (j\geq p)$,
			\item $\omega\in \sum_{(i, j)\in \Delta^t} V_{i}\cap V'_{j}\setminus V_{i-1}\cap V_{j}+V_{i}\cap V'_{j-1}$.
		\end{enumerate}
		
		Similarly, we have 
		\begin{align*}
			[A]_{\Delta}\ast F_h=&(E_h\ast [A^{t}]_{\Delta^t})^t
			= ([B]_{\emptyset}\ast [A^{t}])^{t}\\
			=& (v^{-\sum_{j\leq n}a_{j,h}}e_{B,\emptyset}\ast v^{-\sum_{i<k,j\geq l}a_{i,j}a_{k,l}+\sum_{(i,j)\leq \Delta^t}a_{i,j}}e_{A^t,\Delta^t})^{t},
		\end{align*}
		where $(B,\emptyset)\in \Xi_{m|m,d}$ satisfies $\co (B)=\ro(A^t)$ and $B-E_{h,h+1}$ is a diagonal matrix.
		Besides,      
		\begin{align*}
			e_{B, \emptyset}\ast e_{A^t, \Delta^t}=\sum\limits_{p\in [1,n], a_{h+1,p}>0}\sharp Z_p e_{A^t-E_{i,p}+E_{i+1,p},\Delta'},
		\end{align*}
		where $Z_p$ is the set of all subspaces $U\subset V$ satisfying the following conditions. 
		\begin{enumerate}
			\item $(f,f',\omega)$ is a fixed triple in $\mathscr{X}\times \mathscr{Y}\times V$ lying in the orbit $\mathcal{O}_{A^t+E_{h,p}-E_{h+1,p},\Delta'}$,
			\item $V_{h-1}\subset U\subset V_{h}$ and $\dim{U}=\dim{V_h}-1$,
			\item $ V_h\cap V'_j=U\cap V'_j (j<p), \ V_h\cap V'_{j}\neq U\cap V'_j (j\geq p)$,
			\item $\omega\in \sum_{(i, j)\in \Delta^t} V_{i}\cap V'_{j}\setminus V_{i-1}\cap V_{j}+V_{i}\cap V'_{j-1}$.
		\end{enumerate}
		So we only need to compute the cardinalities of sets $\sharp W_p$ and $\sharp Z_p$.
		
		For any $p\in [1, n]$ and $T\in W_p$, $U\in Z_p$, there exist vectors $y_T\in V_{h+1}$ such that $T=V_h\oplus \mathbb{F}_q y_T$, $y_U\in V_h$ such that $V_h=U\oplus \mathbb{F}_q y_U$. 
		Let $\omega=\sum_{(i, j)\in \Delta'} \omega_{i j}$ where $\omega_{i j}\in V_i\cap V'_{j}\setminus ((V_{i-1}\cap V'_{j})+(V_{i}\cap V'_{j-1}))$.
		
		In $(h)$, we have $(h, i_l), (h+1, i_{l-1})\in \Delta^t$.
		For $ e_{C, \emptyset}\ast e_{A^t, \Delta^t}$, except $\Delta'=\Delta^t$ we have different $\Delta'$ when $i_{l-1}<p\leq i_l$, 
		\begin{align*}
			\Delta'& =\Delta^t_{i_l} =\{(j_k,i_k),\cdots, (h+1, i_l), (j_{l-2},i_{l-2}),\cdots, (j_1,i_1) \} \text{ when }p=i_l,\\
			\Delta'& = \Delta^t_s=\{(j_k,i_k),\cdots, (h,i_l), (h+1, s), (j_{l-2},i_{l-2}),\cdots, (j_1,i_1) \} \text{ when } i_{l-1}<p<i_l.
		\end{align*}
		
		When $\Delta'=\Delta^t$, we have
		\begin{align*}
			\sharp W_p= & \sharp\{T\ |\ T\subset V_h+(V_{h+1}\cap V'_p), \omega_{h+1, i_{l-1}}\notin T\cap V'_{i_{l-1}}+V_{h+1}\cap V'_{i_{l-1}-1} \}\\
			& -\sharp\{T\ |\ T\subset V_h+(V_{h+1}\cap V'_{p-1})\}.
		\end{align*}
		
		For $p>i_{l-1}$, since $T\cap V'_{i_{l-1}}=V_h\cap V'_{i_{l-1}}$, we have $w_{h+1,i_{l-1}}\notin T\cap V'_{i_{l-1}}+V_{h+1}\cap V'_{i_{l-1}-1}$. Therefore,
		\begin{align*}
			\sharp W_p= & \sharp\{T\ |\ T\subset V_h+(V_{h+1}\cap V'_p)\}-\sharp\{T\ |\ T\subset V_h+(V_{h+1}\cap V'_{p-1})\}\\
			= & q^{\sum_{i<p}a_{i,h+1}}\frac{q^{a_{p,h+1}+1}-1}{q-1}.
		\end{align*}
		
		For $p<i_{l-1}$, we have $T\cap V'_{i_{l-1}}\neq V_h\cap V'_{i_{l-1}}$ which means $y_T\in V_h\cap V'_{i_{l-1}}$. If $\omega_{h+1, i_{l-1}}\in T\cap V'_{i_{l-1}}+V_{h+1}\cap V'_{i_{l-1}-1}$, then $\omega_{h+1, i_{l-1}}\in V_{h+1}\cap V'_{i_{l-1}-1}$ which contradict to the choice of $\omega_{h+1, i_{l-1}}$. Therefore, $\sharp W_p$ is the same as the case $p>j_{l-1}$. 
		
		For $p=i_{l-1}$, we have 
		\begin{align*}
			\sharp W_p=& \sharp\{T\ |\ T\subset V_h+(V_{h+1}\cap V'_p), \omega_{h+1, i_{l-1}}\notin T\cap V'_{i_{l-1}}+V_{h+1}\cap V'_{i_{l-1}-1} \}\\
			& -\sharp\{T\ |\ T\subset V_h+(V_{h+1}\cap V'_{p-1})\}\\
			=& q^{\sum_{i<p}a_{i,h+1}}\frac{q^{a_{p,h+1}+1}-1}{q-1}-q^{\sum_{i<p}a_{i,h+1}}\\
			=& q^{\sum_{i<p}a_{i,h+1}+1}\frac{q^{a_{p,h+1}}-1}{q-1}.
		\end{align*}
		
		When $\Delta'=\Delta^t_{i_l}$, let $\tilde{W_p}$ be the set of all subspaces $T\subset V$ that satisfies the conditions $(1),(2)$ and $(3)$. For any $T\in \tilde{W}_p$, we have $T\in W_p$ if and only if $y_T-w_{h+1,i_l}\in V_{h+1}\cap V'_{i_{l-1}}\setminus V_{h}\cap V'_{i_{l-1}}+V_{h+1}\cap V'_{i_{l-1}-1}$.
		Therefore, 
		\begin{align*}
			\sharp W_p=q^{\sum_{i\leq i_{l-1}}a_{i,h+1}}-q^{\sum_{i< i_{l-1}}a_{i,h+1}}.
		\end{align*}
		
		When $\Delta'=\Delta^t_{s}$. For any $T\in \tilde{W}_p$, we have $T\in W_p$ if and only if $y_T-w_{h+1,s}\in V_{h+1}\cap V'_{i_{l-1}}\setminus V_{h}\cap V'_{i_{l-1}}+V_{h+1}\cap V'_{i_{l-1}-1}$. So $\sharp W_p$ is the same as the case $\Delta'=\Delta^t_{i_{l}}$.
		
		For $e_{B, \emptyset}\ast e_{A^t, \Delta^t}$ and $p\neq i_{l-1}$, we have $\Delta'=\Delta^t$. For $p=i_{l-1}$ there are other $\Delta'$, 
		\begin{align*}
			\Delta'& =\Delta^t,\\
			\Delta' & =\Delta^t_{i_l} =\Delta^t\setminus \{(h+1, i_{l-1})\}, \\
			\Delta' & =\Delta^t_s= \{(i_k, j_k), \cdots, (h, i_l), (h+1, s), \cdots, (i_1, j_1)\}, i_{l-2}<s<i_{l-1}.
		\end{align*}
		
		Let $\tilde{Z}_p$ be the set consists of subspaces $U\subset V$ which satisfies the conditions $(1),(2)$ and $(3)$. For $\Delta'=\Delta^t$ and any $U\in \tilde{Z}_p$
		we know $U\cap V'_{i_{l-1}}\subset V_h\cap V'_{i_{l-1}}$, so
		$$\omega_{h+1, j_{m+1}}\in V_{h+1}\cap V'_{i_{l-1}}\setminus{(U\cap V'_{i_{l-1}})+(V_{h+1}\cap V'_{i_{l-1}-1}}).$$ 
		Then $U\in Z_p$ if $\omega_{h, i_l} \in U\cap V'_{i_l}$.
		
		For $p>i_l$, we have $ U\cap V'_{i_l}=V_h\cap V'_{i_l}$, therefore 
		\begin{align*}
			\sharp Z_p=& \sharp\{U\ |\ V_{h-1}+(V_{h}\cap V'_{p-1})\subset U\subset V_{h}\}-\sharp\{U\ |\ V_{h-1}+(V_{h}\cap V'_{p})\subset U\subset V_{h}\}\\
			= & q^{\sum_{i>p}a_{i,h}}\frac{q^{a_{p,h}+1}-1}{q-1}.
		\end{align*}
		
		For $p\leq i_{l-1}$, we have $U\cap V'_{i_{l-1}}\neq V_h\cap V'_{i_{l-1}}$ which shows $y_U\in V_{h+1}\cap V'_{i_{l-1}}$. Therefore, $\sharp Z_p$ is the same as the case $p>i_l$.
		
		For $i_{l-1}<p<i_l$, we have 
		\begin{align*}
			\sharp Z_p=& \sharp\{U\ |\ V_{h-1}+(V_{h}\cap V'_{p-1})\subset U\subset V_{h}, w_{h,i_l}\in U\}\\
			&-\sharp\{U\ |\ V_{h-1}+(V_{h}\cap V'_{p})\subset U\subset V_{h}, w_{h, i_l}\in U\}\\
			= & q^{\sum_{i>p}a_{i,h}-1}\frac{q^{a_{p,h}}-1}{q-1}.
		\end{align*}
		
		For $i=i_l$, we have 
		\begin{align*}
			\sharp Z_p=& \sharp\{U\ |\ V_{h-1}+(V_{h}\cap V'_{p-1})\subset U\subset V_{h}, w_{h,i_l}\in U\}\\
			&-\sharp\{U\ |\ V_{h-1}+(V_{h}\cap V'_{p})\subset U\subset V_{h}, w_{h, i_l}\in U\}\\
			= & q^{\sum_{i>p}a_{i,h}}\frac{q^{a_{p,h}}-1}{q-1}.
		\end{align*}
		
		For $\Delta'=\Delta^t_{i_{l-1}}$ and $\Delta=\Delta^t_s$, we have a decomposition $\omega_{h, i_{l}}=\omega_{1}+\omega_{2}$, where $\omega_{1}\in U\cap V'_{i_{l}}$ and $\omega_{2}\in V_{h}\cap V'_{i_{l}}$. Since $p=j_{m+1}$, we see $y_U\in V_{h+1}\cap V'_{i_{l-1}}\setminus U\cap V'_{i_{l-1}}$. Then for any subspace$U\in \tilde{Z}_p$, $U\in Z_p$ if $\omega_{2}\notin U$. So 
		\begin{align*}
			\sharp  Z_p = & \sharp  \{U\ |\ V_{h-1}+(V_{h}\cap V'_{p-1})\subset U \subset V_{h}\}-\sharp \{U\ |\ V_{h-1}+(V_{h}\cap V'_{p})\subset U \subset V_{h}\}\\
			& - \sharp  \{U\ |\ V_{h-1}+(V_{h}\cap V'_{p-1})\subset U \subset V_{h}, \omega_2\in U\}\\
			& + \sharp  \{U\ |\ V_{h-1}+(V_{h}\cap V'_{p})\subset U \subset V_{h}, \omega_2 \in U\}\\
			= & \frac{q^{\sum_{i\geq p}a_{i, h}+1}-q^{\sum_{i>p}a_{h, i}}}{q-1}-\frac{q^{\sum_{i\geq p}a_{i, h}}-q^{\sum_{i>p}a_{i, h}-1}}{q-1}\\
			= & q^{\sum_{i\geq p}a_{i, h}}-q^{\sum_{i>p}a_{i, h}-1}.
		\end{align*}
		Then $(h)$ follows.
	\end{proof}
	By the surjective algebra homomorphisms $\kappa_{n}$ and $\kappa_{m}$, which sending $\mathbf{E}_i\mapsto E_i$, $\mathbf{F}_i\mapsto F_i$, $\mathbf{H}^{\pm}_a\mapsto H^{\pm}_a$ and $\mathbf{L}\mapsto L$, the space $\mathcal{MV}_{n|m}$ is also equipped with a left $\mathbf{MU}_{n}$-action and a right $\mathbf{MU}_{m}$-action. 
	
	\subsection{Mirabolic Howe duality}
	
	Let 
	$$\mathcal{C}_{n,d}=\{\mathbf{a}=(a_1,\cdots, a_n)\in \mathbb{N}^n | \sum^n_{i=1}a_i=d, a_i>0 \text{ for all } i\}.$$ 
	Any $\mathbf{a}$ corresponds to the parabolic subgroup $P^{\mathbf{a}}$ of $G$ consists of block upper triangular matrices with blocks of sizes $(a_1, \cdots, a_n)$. It is clear that $\mathscr{X}_{n,d}\cong \bigsqcup_{\mathbf{a}\in \mathcal{C}_{n,d}}G/P^{\mathbf{a}}$. Let $B$ be the set of all upper triangular matrices in $G$ and $P_n=\bigsqcup_{\mathbf{a}\in \mathcal{C}_{n,d}}P^{\mathbf{a}}$, $P_m=\bigsqcup_{\mathbf{b}\in \mathcal{C}_{m,d}}P^{\mathbf{b}}$.
	
	Let $G'=G\ltimes V=\{ \begin{pmatrix}
		1& 0\\
		v& g
	\end{pmatrix} | g\in G, v\in V\}$ and $\mathbb{K}(v)[G']$ be the group algebra of it. For all $\mathbf{a}\in \mathcal{C}_{n,d}$, group $P^{\mathbf{a}}$ and $B$ can be regard as subgroups of $G'$ by inclusion $x\mapsto (x,0)$, where $x\in B, P^{\mathbf{a}}$. Hence, we can define the following idempotents in $\mathbb{K}(v)[G']$,
	\begin{align*}
		e_B&=\frac{1}{\lvert B \rvert}\sum_{b\in B}b,\\ 
		e_{P^{\mathbf{a}}}&=\frac{1}{\lvert P^{\mathbf{a}} \rvert}\sum_{p\in P^{\mathbf{a}}}p, \text{ for any }\mathbf{a} \in \mathcal{C}_{n,d}.
	\end{align*}
	It is easy to see $e_{P^{\mbf a}}e_B=e_Be_{P^{\mbf a}}=e_{P^{\mbf a}}$ for all $\mathbf{a}\in \mathcal{C}_{n,d}$.
	
	Rosso proved that $\mathcal{MH}\cong e_B\mathbb{K}(v)[G']e_B$ and $\mathcal{MV}_{n,d} \cong \bigoplus_{\mathbf{a}\in \mathcal{C}_{n,d}} e_{P^{\mathbf{a}}}\mathbb{K}(v)[G']e_B$ in \cite{R14} and \cite{R18}. We will use them to prove the mirabolic Howe duality.
	
	\begin{thm}\label{mhd}
		For $n\geq m\geq d$, the actions
		\begin{center}
			\begin{tikzcd}
				\mathbf{MU}_n \arrow[r, "\kappa_n", two heads] & {\mathcal{MS}_{n, d}} \arrow[r, "\Phi"] & \End(\mathcal{MV}_{n|m}) & {\mathcal{MS}_{m, d}} \arrow[l, "\Psi"'] & \mathbf{MU}_m \arrow[l, "\kappa_m"', two heads]
			\end{tikzcd}
		\end{center}
		satisfy double centralizer property
		\begin{align*}
			\Phi \circ \kappa_{n}(\mathbf{MU}_n)\cong \End_{\mathbf{MU}_m}(\mathcal{MV}_{n|m}),\\
			\Psi \circ \kappa_{m}(\mathbf{MU}_m)\cong \End_{\mathbf{MU}_n}(\mathcal{MV}_{n|m}).
		\end{align*}
	\end{thm}
	\begin{proof}
		Let $\mathbb{K}(v)(G')$ be the space of all $\mathbb{K}(v)$-valued functions on $G'$ and $\mathrm{Fun}(P_n\backslash G'/P_m)$ be the subspace of $\mathbb{K}(v)(G')$ consist of the functions that are constant on the double coset classes $P_n\backslash G'/P_m$. By a similar argument in \cite[Proposition 2.11]{G11}, we have $\mrm{Fun}(P_n\backslash G'/P_m)\cong \bigoplus_{\substack{\mathbf{a}\in \mathcal{C}_{n,d}\\\mathbf{b}\in \mathcal{C}_{m,d}}} e_{P^{\mathbf{a}}}\mathbb{K}(v)[G']e_{P^{\mathbf{b}}}$.
		
		Consider the map between the set of double cosets and the set of $G$-orbit of $\mathscr{X}_{n,d}\times \mathscr{X}_{m,d}\times V$, 
		\begin{align*}
			P_n\backslash G'/P_m&\to G\cdot(\mathscr{X}_{n,d}\times \mathscr{X}_{m,d}\times V)\\
			P_n\begin{pmatrix}
				1& 0\\
				v& g
			\end{pmatrix}P_m&\mapsto G\cdot ( P_n, gP_m,v).
		\end{align*}
		It is clear this map is a bijection. Thus we have $\bigoplus_{\substack{\mathbf{a}\in \mathcal{C}_{n,d}\\\mathbf{b}\in \mathcal{C}_{m,d}}} e_{\mathbf{a}}\mathbb{K}(v)[G']e_{\mathbf{b}}\cong \mrm{Fun}(P_n\backslash G'/P_m) \cong \mathcal{MV}_{n|m}$.
		Therefore, we have
		\begin{align*}
			\mathcal{MV}_{n|m} & \cong \bigoplus_{\substack{\mathbf{a}\in \mathcal{C}_{n,d}\\\mathbf{b}\in \mathcal{C}_{m,d}}} e_{P^{\mathbf{a}}}\mathbb{K}(v)[G']e_{P^{\mathbf{b}}}\cong \bigoplus_{\substack{\mathbf{a}\in \mathcal{C}_{n,d}\\\mathbf{b}\in \mathcal{C}_{m,d}}} e_{P^{\mathbf{a}}}e_B\mathbb{K}(v)[G']e_Be_{P^{\mathbf{b}}}\\
			& \cong \Hom_{e_B\mathbb{K}(v)[G']e_B}(\bigoplus_{\mathbf{b}\in \mathcal{C}_{m,d}}e_B\mathbb{K}(v)[G']e_{P^{\mathbf{b}}}, \bigoplus_{\mathbf{a}\in \mathcal{C}_{n,d}} e_{P^{\mathbf{a}}}\mathbb{K}(v)[G']e_B)\\
			& \cong \Hom_{\mathcal{MH}}(\mathcal{MV}_{m, d}, \mathcal{MV}_{n, d})\\
			& \cong \mathcal{MV}_{n, d} \otimes_{\mathcal{MH}} \mathcal{MV}_{m, d}^{*}.
		\end{align*}
		
		Due to the mirabolic Schur-Weyl duality \cite[Theorem 6.1.6]{FZM}, we have the following decompositions
		$$
		\mathcal{MV}_{n, d}\cong \bigoplus_{i} U_i\otimes M_i, \quad \mathcal{MV}_{m, d}^{*}\cong \bigoplus_{i} M_i^*\otimes V_i^*
		$$
		where $U_i$ and $V_i$ runs over all left simple module of $\mathcal{MS}_{n,d}$ and $\mathcal{MS}_{m, d}$, respectively, and $M_i's$ are certain right simple modules of $\mathcal{MH}$ up to an isomorphism. Therefore,
		\begin{align*}
			\mathcal{MV}_{n|m} & \cong \mathcal{MV}_{n, d} \otimes_{\mathcal{MH}} \mathcal{MV}_{m, d}^{*}\\
			& \cong \bigoplus_{i} U_i\otimes M_i \bigotimes_{\mathcal{MH}} \bigoplus_{i} M_i^*\otimes V_i^*\\
			& \cong \bigoplus_{i} U_i\otimes V_i^{*}.
		\end{align*}
		Then we have 
		\begin{align*}
			\End_{\mathcal{MS}_{m, d}}(\mathcal{MV}_{n|m})& \cong  \End_{\mathcal{MS}_{m, d}}(\bigoplus_{i} U_i\otimes V_i^{*})\\
			& \cong \bigoplus_{i} \End(U_i)\otimes \textrm{id}_{V^*_i}\cong \bigoplus_{i}\End(U_i).
		\end{align*}
		
		By \cite{R14}, the mirabolic Hecke algebra $\mathcal{MH}$ is a semisimple algebra. Due to Theorem \ref{msw-s}, the mirabolic quantum Schur algebra $\mathcal{MS}_{n,d}$ is also a semisimple algebra, so are the quotient $\Phi(\mathcal{MS}_{n,d})$. By Wedderburn-Artin Theorem, we have $\Phi(\mathcal{MS}_{n,d})\cong \bigoplus_{i}\End(U_i)$, then $$\Phi( \mathcal{MS}_{n,d})\cong \End_{\mathcal{MS}_{m,d}}(\mathcal{MV}_{n|m}).$$
		
		The proof for $\Psi(\mathcal{MS}_{m,d})$ is similar and the theorem follows. Then the statement in the theorem hold.
	\end{proof}

	\bibliographystyle{alpha}
	
\end{document}